\pdfoutput=1

\documentclass[12pt]{amsart}

\usepackage{mathptmx}

\usepackage{amsmath}
\usepackage{amssymb}
\usepackage{mathrsfs}
\usepackage[left=2.65cm,right=2.65cm,top=3cm,bottom=3cm]{geometry}

\usepackage[initials]{amsrefs}

\numberwithin{equation}{section}
\newtheorem{theorem}{Theorem}[section]
\newtheorem{lemma}[theorem]{Lemma}

\newtheorem{proposition}[theorem]{Proposition}
\newtheorem{corollary}[theorem]{Corollary}

\theoremstyle{remark} 
\newtheorem{remark}[theorem]{Remark}

\theoremstyle{definition} 
\newtheorem{definition}[theorem]{Definition}

\title[Concave-convex nonlinearities]{Concave-convex nonlinearities
for some nonlinear fractional equations involving the Bessel operator}
\author{Simone Secchi} 
\address{Dipartimento di Matematica e
Applicazioni, Universit\`a degli Studi di Milano Bicocca, via Cozzi
55, 20125 Milano, Italy} 
\email{Simone.Secchi@unimib.it}
\date{\today}

\begin{document}

\begin{abstract} 
	We prove some existence results for a class of
nonlinear fractional equations driven by a nonlocal operator.
\end{abstract}
\maketitle

\tableofcontents

\section{Introduction}

In this paper we provide some existence results for a class of
nonlinear fractional equations of the form
\begin{equation} \label{eq:1} \left( I-\Delta \right)^{\alpha} u +
\lambda V(x) u= f(x,u)+\mu \xi(x)|u|^{p-2}u \quad \text{in
$\mathbb{R}^N$}
\end{equation} 
where $0<\alpha <1$, $V \colon \mathbb{R}^N \to
\mathbb{R}$, $f \colon \mathbb{R}^N \times \mathbb{R} \to \mathbb{R}$
are continuous functions, $\xi \colon \mathbb{R}^N \to (0,+\infty)$
belongs to $L^{2/(2-p)}(\mathbb{R}^N)$, $\lambda >0$, $\mu >0$ and
$1<p<2$.

The operator $(I-\Delta)^\alpha$ is related to the pseudo-relativistic
Schr\"odinger operator $(m^2-\Delta)^{1/2} - m$ ($m > 0$) and recently
a lot of attention is paid to equations involving it.  We refer to
\cites{A-15,A-16-1,A-16-2,CingolaniSecchi1,CingolaniSecchi2,CZN1,CZN2,Fall,FelmerVergara,Se,Tan}
and references therein for more details and physical context of
$(1-\Delta)^\alpha$.  In these papers, the authors study the existence
of nontrivial solution and infinitely many solutions for the equations
with $(m^2 I - \Delta)^\alpha$ and various nonlinearities. We remark
that the hardest issue in dealing with this operator is the lack of
scaling properties: there is no standard group action under which
$(I-\Delta)^\alpha$ behaves as a local differential operator.

In particular, Felmer \emph{et al.} proved in \cite{FelmerVergara} prove the existence of positive solution of $(I-\Delta)^\alpha u = f(x,u)$
under the following conditions on $f(x,s)$: 

\begin{itemize}
\item[(i)] $f \in C(\mathbb{R}^N  \times \mathbb{R}, \mathbb{R})$.  
\item[(ii)] For all $x \in \mathbb{R}^N$, $f(x,s) \geq 0$ if $s \geq 0$ and $f(x,s) = 0$ if $s \leq 0$. 
\item[(iii)] The function $s \mapsto s^{-1}f(x,s)$ is increasing in $(0,\infty)$ for all $x \in \mathbb{R}^N $. 
\item[(iv)] There are $1 < p < 2_\alpha^\ast -1 = (N+2\alpha) / (N-2\alpha)$ and $C>0$ such that 
$|f(x,s) | \leq C |s|^{p}$ for all $(x,s) \in \mathbb{R}^N \times \mathbb{R}$. 
\item[(v)] There exists a $\mu > 2$ such that 
$0 < \mu F(x,s) \leq s f(x,s)$ for all $(x,s) \in \mathbb{R}^N \times (0,\infty)$ where 
$F(x,s) := \int_0^s f(x,t) \, d t$. 
\item[(vi)] There exist continuous functions $\bar{f}(s)$ and $a(x)$ such that 
$\bar{f}$ satisfies (i)--(v) and $0 \leq f(x,s) - \bar{f}(s) \leq a(x) 
(|s| + |s|^p)$ for all $(x,s) \in \mathbb{R}^N  \times \mathbb{R}$, $a(x) \to 0$ as $|x| \to \infty$ 
and 
\[
\mathcal{L}^N( \{ x \in \mathbb{R}^N  \ |\  \text{$f(x,s) > \bar{f}(s)$ for all $s>0$}  \} ) > 0.
\]
\end{itemize}
Earlier this year, Ikoma proved in \cite{Ikoma} the existence of a solution to the equation $(I-\Delta)^\alpha u = f(u)$ under the so-called Berestycki-Lions assumptions on the nonlinerity $f$. As a consequence, Ikoma proved also the existence of a solution to the non-autonomous equation $(I-\Delta)^\alpha u = f(x,u)$ with $f(x,s) = - V(x) s + g(x,s)$ and suitable assumptions on $V$ and $g$.

\medskip

We continue the investigation initiated in \cite{Se}, and we present
two existence results. We fix our standing assumptions on the nonlinear term $f$:
\begin{itemize}
	\item[(f1)] $|f(x,u)| \leq c ( 1+|u|^{q-1})$ for some $q
\in (2,2_\alpha^*)$, where $2_\alpha^* = 2N/(N-2\alpha)$;
	\item[(f2)] $f(x,u)=o(|u|)$ as $u \to 0$ uniformly with
respect to $x \in \mathbb{R}^N$;
	\item[(f3)] there exists a constant $\vartheta > 2$ such that
$0 < \vartheta F(x,u) \leq u f(x,u)$ for every $x \in
\mathbb{R}^N$ and $u \neq 0$, where $F(x,u)=\int_{0}^{u} f(x,s)\, ds$.
\end{itemize}
 The first, Theorem \ref{th:2.3}, shows that
\eqref{eq:1} possesses, for every $\lambda>0$ and $\mu>0$, at least a
nontrivial solution if $V$ has some coercivity property. The second,
Theorem \ref{th:3.1}, shows that \eqref{eq:1} possesses at least two
nontrivial solutions if $\lambda>0$ is sufficiently large and $\mu>0$
is sufficiently small, but without any coercivity on the potential
$V$. Our present assumptions are weaker than those in \cite{Se}, and
more general models are allowed.

In both cases we use variational methods to perform our analysis. The
fact that \eqref{eq:1} is set in the whole $\mathbb{R}^N$ introduces a
natural lack of compactness that must be overcome before applying
Critical Point Theory.

\bigskip


\begin{minipage}{4in} \textbf{Notation}
	\begin{enumerate}
        \item The letters $c$ and $C$ will stand for a generic positive
          constant that may vary from line to line.
        \item The operator $D$ will be reserved for the (Fr\'{e}chet)
          derivative.
        \item The symbol $\mathcal{L}^N$ will be reserved for the
          Lebesgue $N$-dimensional measure.
        \item The Fourier transform of a function $f$ will be denoted
          by $\mathcal{F}u$.
        \item For a real-valued function $V$, we use the notation $V^b = \left\{ x \in \mathbb{R}^N \mid V(x)<b \right\}$.
	\end{enumerate}
\end{minipage}

\section{Preliminaries and functional setting}
For
$\alpha>0$ we introduce the \emph{Bessel function space}
\[ 
L^{\alpha,2}(\mathbb{R}^N) = \left\{ f \colon f=G_\alpha \star g \
\text{for some $g \in L^2(\mathbb{R}^N)$} \right\},
\] 
where the Bessel convolution kernel is defined by
\begin{equation*} 
G_\alpha (x) = \frac{1}{(4 \pi )^{\alpha
/2}\Gamma(\alpha/2)} \int_0^\infty \exp \left( -\frac{\pi}{t} |x|^2
\right) \exp \left( -\frac{t}{4\pi} \right) t^{\frac{\alpha - N}{2}-1}
\, dt
\end{equation*} 
The norm of this Bessel space is $\|f\| = \|g\|_2$ if
$f=G_\alpha \star g$. The operator $(I-\Delta)^{-\alpha} u = G_{2\alpha}
\star u$ is usually called Bessel operator of order $\alpha$.

In Fourier variables the same operator reads
\begin{equation*} 
G_\alpha = \mathcal{F}^{-1} \circ
\left( \left(1+|\xi|^2 \right)^{-\alpha /2} \circ \mathcal{F} \right),
\end{equation*} 
so that
\[ 
\|f\| = \left\| (I-\Delta)^{\alpha /2} f \right\|_2.
\] 
For more detailed information, see \cites{Adams, Stein} and the
references therein.

In the paper \cite{Fall} the pointwise formula
\begin{equation*} 
(I-\Delta)^\alpha u(x) = c_{N,\alpha}
\operatorname{P.V.} \int_{\mathbb{R}^N}
\frac{u(x)-u(y)}{|x-y|^{\frac{N+2\alpha}{2}}}
K_{\frac{N+2\alpha}{2}}(|x-y|) \, dy + u(x)
\end{equation*} 
was derived for functions $u \in C_c^2(\mathbb{R}^N)$.
Here $c_{N,\alpha}$ is a positive constant depending only on $N$ and
$\alpha$, P.V. denotes the principal value of the singular integral,
and $K_\nu$ is the modified Bessel function of the second kind with
order $\nu$ (see \cite{Fall}*{Remark 7.3} for more details).  However
a closed formula for $K_\nu$ is not known.

\smallskip

We summarize the embedding properties of Bessel spaces. For the proofs
we refer to \cite{Felmer}*{Theorem 3.1}, \cite{Stein}*{Chapter V,
Section 3} and \cite{Strichartz}*{Section 4}.

\begin{theorem}
	\begin{enumerate}
		\item $L^{\alpha,2}(\mathbb{R}^N) =
W^{\alpha,2}(\mathbb{R}^N) = H^\alpha (\mathbb{R}^N)$.
		\item If $\alpha \geq 0$ and $2 \leq q \leq
2_\alpha^*=2N/(N-2\alpha)$, then $L^{\alpha,2}(\mathbb{R}^N)$ is
continuously embedded into $L^q(\mathbb{R}^N)$; if $2 \leq q <
2_\alpha^*$ then the embedding is locally compact.
		\item Assume that $0 \leq \alpha \leq 2$ and $\alpha >
N/2$. If $\alpha -N/2 >1$ and $0< \mu \leq \alpha - N/2-1$, then
$L^{\alpha,2}(\mathbb{R}^N)$ is continuously embedded into
$C^{1,\mu}(\mathbb{R}^N)$. If $\alpha -N/2 <1$ and $0 < \mu \leq
\alpha -N/2$, then $L^{\alpha,2}(\mathbb{R}^N)$ is continuously
embedded into $C^{0,\mu}(\mathbb{R}^N)$.
	\end{enumerate}
\end{theorem}
\begin{remark} 
	Although the Bessel space $L^{\alpha,2}(\mathbb{R}^N)$
is topologically undistinguishable from the So\-bo\-lev fractional space
$H^\alpha(\mathbb{R}^N)$, we will not confuse them, since our equation
involves the Bessel norm.
\end{remark}
The next result will be useful in the sequel. We define
\begin{equation*} 
\mathscr{F}(x,u) = \frac{1}{2} f(x,u)
- F(x,u)
\end{equation*} 
for $(x,u) \in \mathbb{R} \times \mathbb{R}$.
\begin{lemma} \label{lem:1.1} 
	For every $\tau \in \left( \max \left\{
1, \frac{N}{2\alpha} \right\},\frac{q}{q-2} \right)$ there exists $R >
0$ such that $|u| \geqslant R$ implies
	\[ \frac{|f(x,u)|^\tau}{|u|^\tau} \leq \mathscr{F}(x,u).
	\]
\end{lemma}
\begin{proof} 
	It follows from (f1)--(f3) that $|f(x,u)| \leq C
|u|^{q-1}$ if $|u|$ is large enough. Now,
\begin{align*}
F(x,u) \leq \frac{1}{\vartheta} u f(x,u) + \frac{1}{2} u f(x,u) - \frac{1}{2}u f(x,u),
\end{align*}
or
\begin{align*}
\left( \frac{1}{2} - \frac{1}{\vartheta} \right) u f(x,u) \leq \mathscr{F}(x,u).
\end{align*}
We claim that, for $|u|$ large enough, 
\begin{align*}
\frac{|f(x,u)|^\tau}{|u|^\tau} \leq \left( \frac{1}{2} - \frac{1}{\vartheta} \right) u f(x,u) .
\end{align*}
It suffices to prove that
\begin{align*}
\frac{|f(x,u)|^{\tau-1}}{|u|^{\tau+1}} \leq \frac{1}{2} - \frac{1}{\vartheta};
\end{align*}
recalling that $|f(x,u)|^{\tau-1} \leq C^{\tau-1} |u|^{(q-1)(\tau-1)}$ for $|u|$ large enough, we deduce that
\[
\frac{|f(x,u)|^{\tau-1}}{|u|^{\tau+1}}  \leq C^{\tau-1} \frac{|u|^{(\tau-1)(q-1)}}{|u|^{\tau+1}} \leq  \frac{1}{2} - \frac{1}{\vartheta} 
\]
for $|u|$ large, since $\tau+1 > (\tau-1)(q-1)$.
\end{proof}

\section{Coercive electric potentials}

In this section we always deal with a \emph{fixed} value of
$\lambda>0$. The lack of compactness of the variational problem
associated to \eqref{eq:1} will be overcome by the following
assumption on the potential $V$.
\begin{definition} 
	We say that $V \colon \mathbb{R}^N \to \mathbb{R}$
is a coercive electric potential if
\begin{itemize}
		\item[(V1)] $\operatorname{ess\,inf}_{x \in
\mathbb{R}^N} V(x) > 0$;
		\item[(V2)] $\lim\limits_{|y| \to +\infty}
\int_{B(y,1)} \frac{dx}{V(x)}=0$, where $B(y,1) = \left\{ x \in
\mathbb{R}^N \mid |y-x|<1 \right\}$.
\end{itemize}
\end{definition} 
The term \emph{coercive} has been used because the
usual coercivity condition
\[ 
\lim_{|x| \to +\infty} V(x) = +\infty
\] 
immediately implies (V2).
\begin{remark} 
	Of course the choice of $B(y,1)$ is fairly arbitrary:
any ball of fixed radius $r>0$ would do the same job.
\end{remark} Define the weighted Sobolev space
\[ 
H = \left\{ u \in L^{\alpha,2}(\mathbb{R}^N) \mid
\int_{\mathbb{R}^N} V(x) |u(x)|^2 \, dx < +\infty \right\}
\] 
equipped with the norm
\begin{equation} \label{eq:norm}
\|u\|_H^2 = \int_{\mathbb{R}^N} \left| (I-\Delta)^{\alpha/2} u
\right|^2 \, dx + \int_{\mathbb{R}^N} V(x) |u(x)|^2 \, dx.
\end{equation}
Since the norm $u \mapsto \int_{\mathbb{R}^N} \left|
(I-\Delta)^{\alpha/2} u \right|^2 \, dx$ already contains the $L^2$
norm, we can allow the inequality $V > 0$ to be true up to a subset of
zero Lebesgue measure. In particular $V$ may vanish on a curve, but
not on an open set. 
Furthermore, equation \eqref{eq:norm} defines a norm even if $V$ is negative: more precisely, it is enough to assume that
\[
V(x) +\inf_{u \in L^{\alpha,2}(\mathbb{R}^N) \setminus \{0\}} \frac{\int_{\mathbb{R}^N} \left| \left( I - \Delta \right)^{\alpha/2}  u\right|^2 \, dx}{\int_{\mathbb{R}^N} |u|^2 \, dx} >0
\]
for every $x \in \mathbb{R}^N$. 
\begin{definition} \label{def:3.3} 
	We say that $u \in
L^{\alpha,2}(\mathbb{R}^N)$ is a weak solution to (\ref{eq:1}) if
	\begin{multline*} 
	\int_{\mathbb{R}^N} (I-\Delta)^{\alpha /2} u
\ (I-\Delta)^{\alpha/2} v \, dx + \int_{\mathbb{R}^N} \lambda V(x)
u(x) v(x) \, dx \\ 
= \int_{\mathbb{R}^N} f(x,u(x))v(x) \, dx +
\int_{\mathbb{R}^N} \xi(x) |u|^{p-2}uv \, dx
	\end{multline*} 
	for all $v \in L^{\alpha,2}(\mathbb{R}^N)$,
or, equivalently,
	\begin{multline*} 
	\int_{\mathbb{R}^N} (1+|\xi|^2)^\alpha
\mathcal{F}u(\xi) \mathcal{F}v(\xi) \, d\xi + \int_{\mathbb{R}^N}
\lambda V(x) u(x) v(x) \, dx \\ 
= \int_{\mathbb{R}^N} f(x,u(x))v(x) \,
dx + \int_{\mathbb{R}^N} \xi(x) |u|^{p-2}uv \, dx.
	\end{multline*}
\end{definition}
\begin{remark} 
	For a general measurable subset $\Omega$ of
$\mathbb{R}^N$, the Bessel space $L^{\alpha,2}(\Omega)$ is defined as
the set of restrictions to $\Omega$ of functions from
$L^{\alpha,2}(\mathbb{R}^N)$. This will be useful in the following
Proposition.
\end{remark}
\begin{proposition} \label{prop:2.2} 
	If $V$ is a compact electric
potential and $2\leq q<2_\alpha^*$, then the space $H$ is
compactly embedded into $L^q(\mathbb{R}^N)$.
\end{proposition}
\begin{proof} 
	In the proof we will discard the set where $V$ vanishes, 
	since it has zero measure. We follow closely \cite{MR0481616}. Let
$\{u_n\}_n$ be a sequence from $H$ such that $u_n \rightharpoonup 0$
as $n \to +\infty$. For some $M>0$, we have
	\[ 
	\int_{\mathbb{R}^N}V(x)|u(x)|^2 \, dx \leq M \quad\text{for
all $n\in \mathbb{N}$}.
	\] 
	We define
	\begin{align*} 
	\Theta_m &= \left\{ A \subset \mathbb{R}^N \mid
\text{$A$ is measurable and $\lim_{|x| \to +\infty} \mathcal{L}^N (A
\cap B(x,1))=0$} \right\}\\ 
\Theta_0 &= \left\{ \Omega \in \Theta_m
\mid \text{$\Omega$ is open}\right\}.
	\end{align*} 
	It follows from assumption (V1) that $H$ embeds
continuously into $L^{\alpha,2}(\mathbb{R}^N)$, and therefore the
restriction of $u_n$ to $\Omega$ converges weakly to zero in
$L^{\alpha,2}(\Omega)$ for any $\Omega \in \Theta_0$. We can refer to
Theorem 2.4 of \cite{MR0312241} and conclude that $(u_n)_{|\Omega} \to
0$ strongly in $L^2(\Omega)$ for every $\Omega \in \Theta_0$.
	
	Now pick any $\varepsilon>0$. We compute for all $n \in
\mathbb{N}$:
	\begin{equation*} \left\| u_n \right\|_{L^2(\mathbb{R}^N)}^{2}
= \int_{\mathbb{R}^N} \frac{1}{V(x)} \left| u_n(x) \right|^2 V(x) \,
dx \leq M \varepsilon + \int_{V^{1/\varepsilon}} \left| u_n(x)
\right|^2 \, dx,
	\end{equation*} 
	where
	\[ 
	V^{1/\varepsilon} = \left\{ x \in \mathbb{R}^N \mid V(x) <
1/\varepsilon \right\}.
	\] 
It follows from assumption (V2) that $V^{1/\varepsilon} \in
\Theta_m$. We claim that there exists an open set $\Omega_\varepsilon
\in \Theta_0$ such that $V^{1/\varepsilon} \subset \Omega_\varepsilon$. If
this is the case, we recall that $(u_n)_{|\Omega_\varepsilon}$ converges to zero strongly in $L^2(\Omega_\varepsilon)$ and conclude easily that
	\[ 
	\limsup_{n \to +\infty} \left\| u_n
\right\|_{L^2(\mathbb{R}^N)}^{2} \leq M \varepsilon +
\int_{\Omega_\varepsilon} \left| u_n(x) \right|^2 \, dx \leq (M+1)
\varepsilon.
	\] 
To prove the claim, we introduce a countable family
$\{\mathscr{O}_k\}_k$ of open sets such that $\mathscr{O}_1 = B(0,1)$,
$V^{1/\varepsilon}\cap (\overline{B(0,k+1)}\setminus B(0,k)) \subset
\mathscr{O}_{k+1}$ and
	\[ 
	\mathcal{L}^N \left( \mathscr{O}_{k+1} \setminus \left(
V^{1/\varepsilon} \cap (\overline{B(0,k+1)}\setminus B(0,k)) \right)
\right) < \frac{1}{k}.
	\] 
	We put $\Omega_\varepsilon = \bigcup_{k=1}^\infty
\mathscr{O}_k$. Now,
	\[ 
	\mathcal{L}^N (B(x,1) \cap \Omega_\varepsilon) \leq
\mathcal{L}^N( B(x,1) \cap V^{1/\varepsilon}) + \mathcal{L}^N
((\Omega_\varepsilon \setminus V^{1/\varepsilon}) \cap B(x,1)).
	\] We need to check that $\lim_{|x| \to +\infty} \mathcal{L}^N
((\Omega_\varepsilon \setminus V^{1/\varepsilon}) \cap B(x,1)) =0$. Let $x
\in \mathbb{R}^N$ and let $n_1$ be the largest integer such that $n_1
\leq |x|-1$. We deduce from the properties of
$\{\mathscr{O}_k\}_k$ that
	\[ \mathcal{L}^N ((\Omega_\varepsilon \setminus V^{1/\varepsilon})
\cap B(x,1)) \leq \mathcal{L}^N \left( \bigcup_{k=n_1+1}^{n_1+3}
\mathscr{O}_k \setminus V^{1/\varepsilon} \right) < \frac{3}{n_1} <
\frac{3}{|x|-2}.
	\] The claim is proved.
	
	So far we have shown that $H$ is compactly embedded into
$L^2(\mathbb{R}^N)$. If $q \in [2,2_\alpha^*)$, we recall that $H$ is
continuously embedded into $L^{2_\alpha^*}(\mathbb{R}^N)$ and the
compactness of the embedding into $L^q(\mathbb{R}^N)$ follows from
standard interpolation inequalities in Lebesgue spaces.
\end{proof}
\begin{remark}
We notice that $V$ is a coercive electric potential if, and only if, 
\[
\lim_{|x| \to +\infty} \mathcal{L}^N \left( B(x,1) \cap V^b \right) =0
\]
for every $b \in \mathbb{R}$. 
\end{remark}
\begin{theorem} \label{th:2.3} Assume that (f1), (f2), (f3), (V1) and
(V2) hold. For every $\lambda>0$ there exists $\mu_0>0$ such that for 
every $\mu \in (0,\mu_0)$, there exists at least a nontrivial solution
to \eqref{eq:1}.
\end{theorem} We will prove Theorem \ref{th:2.3} by variational
methods. First of all, we associate to equation \eqref{eq:1} the Euler
functional $\Phi \colon H \to \mathbb{R}$ defined by
\begin{equation} \label{eq:2.1} \Phi (u) = \frac{1}{2}
\int_{\mathbb{R}^N} \left| (I-\Delta)^{\alpha/2}u \right|^2 \, dx +
\frac{\lambda}{2} \int_{\mathbb{R}^N} V(x) |u(x)|^2 \, dx - \Psi (u),
\end{equation} where
\begin{equation*} \Psi (u) = \int_{\mathbb{R}^N} F(u(x))\, dx +
\frac{\mu}{p} \int_{\mathbb{R}^N} \xi(x) |u(x)|^{p} \, dx.
\end{equation*} 
It is easy to check that $\Phi$ is continuously
differentiable on $H$ under our assumptions. Moreover,
\begin{align*}
D\Phi (u)[v] &= \int_{\mathbb{R}^N}(I-\Delta)^{\alpha/2}u (I-\Delta)^{\alpha/2} v \, dx + \int_{\mathbb{R}^N} \lambda V(x) u(x) v(x) \, dx \\
D\Psi(u)[v] &= \int_{\mathbb{R}^N} f(x,u(x))v(x)\, dx + \mu \int_{\mathbb{R}^N} \xi(x) |u(x)|^{p-2} u(x) v(x) \, dx,
\end{align*}
so that weak solutions
to \eqref{eq:1} correspond to critical points of $\Phi$ via Definition
\ref{def:3.3}.

We will check that $\Phi$ satisfies the geometric assumptions of the Mountain Pass Theorem, see \cite{AmRa}.
\begin{lemma} \label{lem:2.4} Let us retain the assumption of Theorem
\ref{th:2.3}.
  \begin{itemize}
    \item[(i)] There exist three positive constants $\mu_0$, $\varrho$
and $\eta$ such that $\Phi(u) \geqslant \eta$ for all $u \in H$ with
$\|u\|_H =\varrho$ and all $\mu \in (0,\mu_0)$.
    \item[(ii)] Let $\varrho>0$ be the number constructed in step
(i). There exists $e \in H$ such that $\|e\|_H > \varrho$ and $\Phi(e)
< 0$ for all $\mu \geqslant 0$.
\end{itemize}
\end{lemma}
\begin{proof} Let us prove (i). For any fixed $\varepsilon>0$,
assumptions (f1) and (f2) imply that there is a positive constant
$C_\varepsilon$ such that
\begin{equation} \label{eq:3.2.1} |F(x,u)| \leq \frac{\varepsilon}{2}
|u|^2 + \frac{C_\varepsilon}{q}|u|^q
\end{equation} for all $x \in \mathbb{R}^N$ and $u \in
\mathbb{R}$. Integrating and using the Sobolev inequality, we get
\begin{align*} \int_{\mathbb{R}^N} F(x,u)\, dx &\leq
\frac{\varepsilon}{2} \int_{\mathbb{R}^N} |u(x)|^2 \, dx +
\frac{C_\varepsilon}{q} \int_{\mathbb{R}^N} |u(x)|^q \, dx \\
&\leq C \left( \frac{\varepsilon}{2} \|u\|_H^2 +
\frac{C_\varepsilon}{q} \|u\|_H^q \, dx \right).
\end{align*} Therefore
\begin{align*} \Phi(u) &= \frac{1}{2} \|u\|_H^2 - \int_{\mathbb{R}^N}
F(x,u(x))\, dx - \frac{\mu}{p} \int_{\mathbb{R}^N} \xi(x)|u(x)|^p \,
dx \\ &\geqslant \frac{1}{2} \|u\|_H^2 -C \left( \frac{\varepsilon}{2}
\|u\|_H^2 + \frac{C_\varepsilon}{q} \|u\|_H^q \, dx \right) -
\frac{\mu}{p} C\|\xi\|_{L^{2/(2-p)}(\mathbb{R}^N)} \|u\|_H^p \\ &=
\|u\|_H^p \left( \frac{1}{2} \left( 1- C \varepsilon \right)
\|u\|_H^{2-p} - \frac{C C_\varepsilon}{q} \|u\|_H^{q-p} -
\frac{\mu}{p} C\|\xi\|_{L^{2/(2-p)}(\mathbb{R}^N)} \right).
\end{align*} We select (for instance) $\varepsilon = \frac{1}{2C}$ and
maximize the function
\[ g(t) = \frac{1}{4} t^{2-p} - \frac{C C_\varepsilon}{q}t^{q-p}
\] for $t \geqslant 0$. It is an exercise to check that the maximum is
attained at some $\varrho>0$ where $g(\varrho)>0$. We conclude by
selecting $\mu_0>0$ so small that $g(\varrho) - \frac{\mu_0}{p}
C\|\xi\|_{L^{2/(2-p)}(\mathbb{R}^N)} >0$.

To prove (ii), we recall \eqref{eq:3.2.1}. By assumption (f3), for
some constant $c>0$ we have
\[ F(x,u) \geqslant c \left( |u|^\vartheta - |u|^2 \right)
\] for all $(x,u) \in \mathbb{R}^N \times \mathbb{R}$. Given $u \in H$
and $t>0$, we compute
\begin{align*} 
\Phi(tu) &= \frac{t^2}{2} \|u\|_H^2 -
\int_{\mathbb{R}^N} F(x,tu(x))\, dx - \frac{\mu}{p} t^p
\int_{\mathbb{R}^N} \xi(x)|u(x)|^p \, dx \\ 
&\leq \frac{t^2}{2}
\|u\|_H^2 - c t^\vartheta \int_{\mathbb{R}^N} |u(x)|^\vartheta \, dx +
ct^2 \int_{\mathbb{R}^N} |u(x)|^2 \, dx - \frac{\mu}{p} t^p
\int_{\mathbb{R}^N} \xi(x)|u(x)|^p \, dx.
\end{align*} 
Recalling that $1<p<2$ and $\vartheta>2$, we can let $t
\to +\infty$ and deduce that $\Phi(tu) \to -\infty$. The conclusion is
now immediate.
\end{proof}
We can now prove the main result of this section.
\begin{proof}[Proof of Theorem \ref{th:2.3}] For $0<\mu < \mu_0$, the
functional $\Phi$ satisfies the geometric assumptions of the Mountain
Pass Theorem. As a consequence, there exist a Palais-Smale sequence
$\{u_n\}_n$ from $H$, i.e.
	\begin{equation*} \Phi(u_n) \to c, \qquad D\Phi(u_n) \to 0
	\end{equation*} as $n \to +\infty$, where
\[ c = \inf_{\gamma \in \Gamma} \sup_{t \in [0,1]} \Phi(\gamma(t))
\quad\text{and} \quad \Gamma = \left\{ \gamma \in C([0,1],H) \mid
\gamma(0)=0, \ \gamma(1)=e \right\}.
\] We prove that $\{u_n\}_n$ is bounded. Indeed,
\begin{multline*} 1+c+\|u_n\|_H \geqslant \Phi(u_n) -
\frac{1}{\vartheta} D \Phi (u_n)[u_n] \\ = \left( \frac{1}{2} -
\frac{1}{\vartheta} \right) \|u_n\|_H^2 + \int_{\mathbb{R}^N} \left(
\frac{1}{\vartheta} u_n(x) f(x,u_n(x))-F(x,u_n(x)) \right) \, dx \\
{}+ \left( \frac{1}{\vartheta}-\frac{1}{p} \right) \int_{\mathbb{R}^N}
\mu \xi(x) |u_n(x)|^p \, dx.
\end{multline*} Since
\begin{multline*} \left( \frac{1}{p}-\frac{1}{\vartheta} \right) \mu
\int_{\mathbb{R}^N} \xi(x) |u_n(x)|^p \, dx \leq \left(
\frac{1}{p}-\frac{1}{\vartheta} \right) \mu \left( \int_{\mathbb{R}^N}
|\xi(x)|^{\frac{2}{2-p}} \, dx \right)^{\frac{2-p}{2}} \left(
\int_{\mathbb{R}^N} |u_n(x)|^2 \, dx \right)^{\frac{p}{2}} \\ = \left(
\frac{1}{p}-\frac{1}{\vartheta} \right) \mu
\|\xi\|_{L^{2/(2-p)}(\mathbb{R}^N)} \|u_n\|_{L^2(\mathbb{R}^N)}^p
\leq C \left( \frac{1}{p}-\frac{1}{\vartheta} \right) \mu
\|\xi\|_{L^{2/(2-p)}(\mathbb{R}^N)} \|u_n\|_H^p,
\end{multline*} we derive that
\begin{multline*} 1+c+\|u_n\|_H +C \left(
\frac{1}{p}-\frac{1}{\vartheta} \right) \mu
\|\xi\|_{L^{2/(2-p)}(\mathbb{R}^N)} \|u_n\|_H^p \\ \geqslant \left(
\frac{1}{2} - \frac{1}{\vartheta} \right) \|u_n\|_H^2 +
\int_{\mathbb{R}^N} \left( \frac{1}{\vartheta} u_n(x)
f(x,u_n(x))-F(x,u_n(x)) \right) \, dx \\ \geqslant \left( \frac{1}{2}
- \frac{1}{\vartheta} \right) \|u_n\|_H^2.
\end{multline*} Since $1<p<2$, this inequality shows that $\{u_n\}_n$
is a bounded sequence in $H$. By Proposition~\ref{prop:2.2},
$\{u_n\}_n$ converges up to a subsequence (weakly in $H$ and) strongly
in $L^2(\mathbb{R}^N)$ and in $L^q(\mathbb{R}^N)$ to some limit
$u$. Since
\[ \int_{\mathbb{R}^N} \xi(x) |u_n(x)-u(x)|^p \, dx \leq
\|\xi\|_{L^{2/(2-p)}(\mathbb{R}^N} \|u_n-u\|_{L^2(\mathbb{R}^N)}^p,
\] it follows immediately that $\{u_n\}_n$ is relatively compact in
$H$, or, in other words, that $\Phi$ satisfies the Palais-Smale
condition. Hence $u$ is a critical point of $\Phi$, namely a weak
solution to \eqref{eq:1}.
\end{proof}
We conclude this section with a regularity result. We skip its proof, since it is based on arguments that already appear in \cites{FelmerVergara,Ikoma,Se}.
\begin{theorem} \label{th:3.9}
The solution $u \in H$ belongs to $C_b^\beta(\mathbb{R}^N)$ for every $\beta \in (0,2\alpha)$. Here
\[
C_b^\beta(\mathbb{R}^N) = \left\{ u \in C(\mathbb{R}^N) \cap L^\infty(\mathbb{R}^N) \mid \sup_{\substack{x,y \in \mathbb{R}^N \\ x \neq y}} \frac{|u(x)-u(y)|}{|x-y|} < +\infty \right\} \quad \text{if $\beta<1$},
\]
and
\[
C_b^\beta(\mathbb{R}^N) = \left\{ u \in C^1(\mathbb{R}^N) \mid u \in L^\infty(\mathbb{R}^N),\ \nabla u \in L^\infty(\mathbb{R}^N) \cap C_b^{\beta-1}(\mathbb{R}^N)\right\} \quad \text{if $1<\beta<2$}.
\]
\end{theorem}

\section{Solutions for large values of $\lambda$ and small values of
$\mu$}

In this section we solve equation \eqref{eq:1} under weaker conditions
on the electric potential $V$. As we have seen in the previous
section, the compactness of $V$ yields the validity of the
Palais-Smale condition almost for free. We show that we can relax the
assumptions on $V$, provided that the parameter $\lambda$ is
sufficiently large.

Precisely, we assume the following:
\begin{itemize}
	\item[(V3)] $V \geq 0$ on $\mathbb{R}^N$;
	\item[(V4)] for some $b>0$, the Lebesgue measure of the set
$V^b = \left\{ x \in \mathbb{R}^N \mid V(x) < b \right\}$ is finite;
	\item[(V5)] the set\footnote{The notation $A^\circ$ is used to
denote the interior of a set $A$.} $\Omega = \left(V^{-1} \left( \{0\}
\right) \right)^\circ$ is nonempty and has a smooth
boundary. Furthermore, $\overline{\Omega} = V^{-1}(\{0\})$.
\end{itemize}
\begin{theorem} \label{th:3.1} Assume that (V3), (V4), (V5) and (f1),
(f2), (f3) are satisfied. There exist two constants $\lambda_0>0$ and
$\mu_0>0$ such that for every $\lambda > \lambda_0$ and every
$0<\mu<\mu_0$ equation \eqref{eq:1} possesses at least two nontrivial
solutions.
\end{theorem} Again we will prove this result by means of variational
methods. Since $\lambda$ is no longer fixed, we will use the notation
$\Phi_\lambda$ for the Euler functional \eqref{eq:2.1}.

We define the space
\[ \mathscr{H} = \left\{ u \in L^{\alpha,2}(\mathbb{R}^N) \mid
\int_{\mathbb{R}^N} V(x) |u(x)|^2 \, dx < +\infty\right\}
\] endowed with the norm
\[ \|u\|_{\mathscr{H}}^2 = \int_{\mathbb{R}^{N}} \left|
(I-\Delta)^{\alpha/2} \right|^2 \, dx + \int_{\mathbb{R}^N} V(x)
|u(x)|^2 \, dx.
\] For technical reasons, we will need to work with the norm
\[ \|u\|_\lambda^2 = \int_{\mathbb{R}^{N}} \left|
(I-\Delta)^{\alpha/2} \right|^2 \, dx + \int_{\mathbb{R}^N} \lambda
V(x) |u(x)|^2 \, dx,
\] and we will write $\mathscr{H}_\lambda$ to denote the space
$\mathscr{H}$ endowed with the norm $\| \cdot \|_\lambda$.
\begin{lemma} \label{lem:3.2} There exist two constants $\gamma_0>0$
and $\lambda_0>0$ such that for every $\lambda \geqslant \lambda_0$
there results
	\[ \|u\|_{L^{\alpha,2}(\mathbb{R}^N)} \leq \gamma_0
\|u\|_\lambda
	\] for every $u \in \mathscr{H}_\lambda$.
\end{lemma}
\begin{proof} Indeed, we get from the Sobolev embedding theorem
 \begin{multline*} 
 \int_{\mathbb{R}^N} |u(x)|^2 \, dx = \int_{V^b}
|u(x)|^2 \, dx + \int_{\mathbb{R}^N\setminus V^b} |u(x)|^2 \, dx \\
\leq \left(\mathcal{L}^N \left( V^b
\right)\right)^{\frac{2\alpha}{N}} \left( \int_{\mathbb{R}^N}
|u(x)|^{2_\alpha^*} \, dx \right)^{\frac{N-2\alpha}{2}} +
\int_{\mathbb{R}^N\setminus V^b} |u(x)|^2 \, dx \\ 
\leq
\left(\mathcal{L}^N \left( V^b \right)\right)^{\frac{2\alpha}{N}}
\left( \int_{\mathbb{R}^N} |u(x)|^{2_\alpha^*} \, dx
\right)^{\frac{N-2\alpha}{2}} + \frac{1}{\lambda
b}\int_{\mathbb{R}^N\setminus V^b} \lambda V(x) |u(x)|^2 \, dx \\
\leq C \left(\mathcal{L}^N \left( V^b
\right)\right)^{\frac{2\alpha}{N}} \int_{\mathbb{R}^N} \left|
(I-\Delta) u \right|^2 \, dx + \frac{1}{\lambda
b}\int_{\mathbb{R}^N\setminus V^b} \lambda V(x) |u(x)|^2 \, dx,
	 \end{multline*} 
	 and therefore
 \begin{multline*} 
 \int_{\mathbb{R}^N} \left| (I-\Delta)^{\alpha/2} u
\right|^2 \, dx \leq \int_{\mathbb{R}^N} \left|
(I-\Delta)^{\alpha/2} u \right|^2 \, dx + \int_{\mathbb{R}^N} |u(x)|^2
\, dx \\ 
\leq \left(1+ C \left(\mathcal{L}^N \left( V^b
\right)\right)^{\frac{2\alpha}{N}} \right) \left( \int_{\mathbb{R}^N}
\left| (I-\Delta)^{\alpha/2} u \right|^2 \, dx + \int_{\mathbb{R}^N}
\lambda V(x) |u(x)|^2 \, dx \right) \\ = \gamma_0 \left(
\int_{\mathbb{R}^N} \left| (I-\Delta)^{\alpha/2} u \right|^2 \, dx +
\int_{\mathbb{R}^N} \lambda V(x) |u(x)|^2 \, dx \right)
 \end{multline*} 
 whenever
 \[ \lambda \geqslant \lambda_0 = \frac{1}{b} \frac{1}{1+ C
\left(\mathcal{L}^N \left( V^b \right)\right)^{\frac{2\alpha}{N}}}.
 \]
\end{proof}
\begin{corollary} 
	For all $s \in [2,2_\alpha^*)$, there exists a
constant $\gamma_s >0$ such that 
\[
\|u\|_{L^s(\mathbb{R}^N)} \leq
\gamma_s \|u\|_{L^{\alpha,2}(\mathbb{R}^N)}\leq \gamma_0 \gamma_s
\|u\|_\lambda
\]
for every $u \in \mathscr{H}_\lambda$.
\end{corollary}
\begin{proof} It suffices to combine the Sobolev embedding theorem
with Lemma \ref{lem:3.2}.
\end{proof} The mountain-pass geometry of $\Phi_\lambda \colon
\mathscr{H}_\lambda \to \mathbb{R}$ is ensured by Lemma
\ref{lem:2.4}. On the contrary, the Palais-Smale condition is now
harder to prove, since no \emph{coerciveness} assumption on the
electric potential has been made.
\begin{lemma} Suppose that $u_n \rightharpoonup u_0$ in
$\mathscr{H}_\lambda$ as $n \to +\infty$. Then, up to a subsequence,
	\begin{equation} \label{eq:3.1} \Phi_\lambda(u_n) =
\Phi_\lambda (u_n-u_0) + \Phi_\lambda(u_0)+o(1)
	\end{equation} and
	\begin{equation}\label{eq:3.2} D\Phi_\lambda(u_n) =
D\Phi_\lambda (u_n-u_0) + D\Phi_\lambda(u_0)+o(1)
	\end{equation} as $n \to +\infty$. In particular, if
$\{u_n\}_n$ is a Palais-Smale sequence at level $d$, then
	\begin{equation}\label{eq:3.3} \Phi_\lambda(u_n-u_0) =
d-\Phi_\lambda(u_0) +o(1), \quad D\Phi_\lambda(u_n-u_0) = o(1)
	\end{equation} as $n \to +\infty$, up to a subsequence.
\end{lemma}
\begin{proof} From the weak convergence assumption on $u_n$ it follows
that $\|u_n\|_\lambda^2 = \|u_n-u_0\|_\lambda^2 +\|u_0\|_\lambda^2 +
o(1)$. To prove \eqref{eq:3.1} and \eqref{eq:3.2} it will be enough to
check that as $n \to +\infty$
	\begin{align} &\int_{\mathbb{R}^N}
\left(F(x,u_n(x))-F(x,u_n(x)-u_0(x))-F(x,u_0(x)) \right) dx
=o(1) \label{eq:3.4} \\ &\int_{\mathbb{R}^N} \xi(x)\left(|u_n(x)|^p -
|u_n(x) - u_0(x)|^p -|u_0(x)|^p \right) dx =o(1) \label{eq:3.5} \\
&\int_{\mathbb{R}^N} \left(f(x,u_n(x)) -f(x,u_n(x)-u_0(x))-f(x,u_0(x))
\right) \phi(x)\, dx =o(1) \label{eq:3.6}
\end{align} and
\begin{multline} \int_{\mathbb{R}^N} \xi(x) (|u_n(x)|^{p-2}u_n(x)
-|u_n(x)-u_0(x)|^{p-2}(u_n(x)-u_0(x))-|u_0(x)|^{p-2}u_0(x)) \phi(x) \,
dx \\ = o(1) \label{eq:3.7}
\end{multline} for all $\phi \in \mathscr{H}_\lambda$. To prove
\eqref{eq:3.4} we follow the spirit of an idea due to Brezis and Lieb (\cite{BrLi}). We define
$\delta_n = u_n - u_0$ so that $\delta_n \to 0$ in
$L_{\mathrm{loc}}^2(\mathbb{R}^N)$ and observe that for every
$\varepsilon>0$ there exists a constant $C_\varepsilon>0$ such that
\begin{align} |f(x,u)| &\leq \varepsilon |u| + C_\varepsilon
|u|^{q-1} \\ |F(x,u)| &\leq \int_0^1 \left| f(x,tu) \right| |u|\,
dt \leq \varepsilon |u|^2 + C_\varepsilon |u|^{q} \label{eq:3.9}
\end{align} for every $(x,u) \in \mathbb{R}^N \times
\mathbb{R}^N$. Hence
\begin{multline*} 
\left| F(x,\delta_n + u_0) -F(x,\delta_n) \right|
\leq \int_0^1 |f(x,\delta_n + \zeta u_0)| |u_0|\, d\zeta \\
\leq \int_0^1 \left( \varepsilon |\delta_n + \zeta u_0||u_0| +
C_\varepsilon |\delta_n + \zeta u_0|^{q-1} |u_0| \right) d\zeta \\
\leq C \left( \varepsilon |\delta_n||u_0| + \varepsilon |u_0|^2 +
C_\varepsilon |\delta_n|^{q-1} |u_0| + C_\varepsilon 
|u_0|^{q} \right),
\end{multline*} 
and the Young inequality for numbers implies that
\begin{equation*} \left| F(x,\delta_n + u_0) -F(x,\delta_n) \right|
\leq C \left( \varepsilon |\delta_n|^2 + \varepsilon |u_0|^2
+C_\varepsilon |\delta_n|^{q} + C_\varepsilon |u_0|^{q} \right).
\end{equation*} Using \eqref{eq:3.9} we find similarly
\begin{equation*} \left| F(x,\delta_n + u_0) -F(x,\delta_n) - F(x,u_0)
\right| \leq C \left( \varepsilon |\delta_n|^2 + \varepsilon
|u_0|^2 +C_\varepsilon |\delta_n|^{q} + C_\varepsilon |u_0|^{q}
\right).
\end{equation*} We introduce
\[ 
M_n(x) = \left( F(x,\delta_n + u_0) -F(x,\delta_n) - F(x,u_0)
-\varepsilon |\delta_n|^2 -C_\varepsilon |\delta_n|^{q} \right) \lor
0,
\] 
where $a \lor b = \max\{a,b\}$. From the previous estimates it
follows that $0 \leq M_n \leq \varepsilon |u_0|^2 + C_\varepsilon
|u_0|^{q} \in L^1(\mathbb{R}^N)$ for every $n \in \mathbb{N}$. By dominated convergence, $M_n \to 0$ in $L^1(\mathbb{R}^N)$. Therefore
\begin{multline*} 
	\limsup_{n \to +\infty} \int_{\mathbb{R}^N}\left|
F(x,\delta_n(x) + u_0(x)) -F(x,\delta_n(x)) - F(x,u_0(x)) \right| \,
dx \\ 
\leq C \varepsilon \limsup_{n \to +\infty} \left(
\|\delta_n\|_{L^2(\mathbb{R}^N)}^2 +
\|\delta_n\|_{L^q(\mathbb{R}^N)}^q \right),
\end{multline*} 
and \eqref{eq:3.4} follows.  To prove \eqref{eq:3.5}
we simply recall that $\delta_n \to 0$ strongly in
$L_{\mathrm{loc}}^2(\mathbb{R}^N)$. With the aid of the H\"{o}lder
inequality, for any measurable set $A$ we can write
\begin{equation*} 
	\int_A \xi(x) | \delta_n(x) |^p \, dx \leq
\left(\int_A | \xi(x) |^{\frac{2}{2-p}} \, dx \right)^{\frac{2-p}{2}}
\left( \int_A |\delta_n(x)|^2 \, dx \right)^{\frac{p}{2}}.
\end{equation*} 
We take $A=\mathbb{R}^N \setminus B(0,\bar{R})$ for a
sufficiently large radius $\bar{R}$ in such a way that
$\int_{\mathbb{R}^N\setminus B(0,\bar{R})} | \xi(x) |^{\frac{2}{2-p}} \, dx$
is as small as we like. The term $\int_{B(0,\bar{R})} |\delta_n(x)|^2
\, dx$ is small due to the strong local convergence. We have thus proved that
\[ 
\int_{\mathbb{R}^N} \xi(x) |\delta_n(x)|^p \, dx =o(1).
\] 
Since
\[ 
\left| \int_{\mathbb{R}^N} \xi(x) \left( |u_n(x)|^p - |u_0(x)|^p
\right) \, dx \right| \leq \int_{\mathbb{R}^N} \xi(x)
|\delta_n(x)|^p \, dx,
\] 
the proof of \eqref{eq:3.5} is complete. Reasoning in a very
similar way we can also check the validity of \eqref{eq:3.6} and
\eqref{eq:3.7}.  The last part of the Lemma is standard, and we omit
it.
\end{proof}
\begin{lemma} \label{lem:3.5} Assume that (V3), (V4), (V5) and (f1),
(f2), (f3) hold. For some $\Lambda>0$, the functional $\Phi_\lambda$
satisfies the Palais-Smale condition for any $\lambda \geqslant
\Lambda$.
\end{lemma}
\begin{proof} We follow some ideas of \cite{MR2321894}. As in the proof
of Theorem \ref{th:2.3}, any Palais-Smale sequence $\{u_n\}_n$ for
$\Phi_\lambda$ at level $d$ is bounded. Up to a subsequence, we may
assume that $u_n \rightharpoonup u_0$ in $\mathscr{H}_\lambda$ and
$u_n \to u_0$ strongly in $L^r_{\mathrm{loc}}(\mathbb{R}^N)$ for every
$r \in [2,2_\alpha^*)$. Writing again $\delta_n = u_n-u_0$, assumption
(V4) implies that
\begin{equation} \label{eq:3.10} \int_{\mathbb{R}^N} |\delta_n(x)|^2
\, dx \leq \frac{1}{\lambda b} \int_{\mathbb{R}^N \setminus V^b}
\lambda V(x) |\delta_n(x)|^2 \, dx + \int_{V^b} |\delta_n(x)|^2 \, dx
\leq \frac{1}{\lambda b} \| \delta_n\|^2_{\lambda} + o(1).
\end{equation} and remark that
\begin{align*} 
\int_{\mathbb{R}^N} \mathscr{F}(x,\delta_n(x)) \, dx &=
\Phi_\lambda(\delta_n) - \frac{1}{2} D\Phi_\lambda(\delta_n)[\delta_n]
- \left( \frac{1}{2} - \frac{1}{p} \right) \mu \int_{\mathbb{R}^N}
\xi(x) |\delta_n(x)|^p \, dx \\
&= d - \Phi_\lambda(u_0) +o(1)
\end{align*} 
by \eqref{eq:3.3}. Let
\[ 
N_0 = \sup_{n \in \mathbb{N}} \left| \int_{\mathbb{R}^N}
\mathscr{F}(x,\delta_n(x)) \, dx \right|, \quad \sigma = \frac{2\tau}{\tau-1} \in (2,2_\alpha^*).
\] 
From the
H\"{o}lder inequality and Lemma \ref{lem:1.1},\footnote{$R>0$ is the number constructed in Lemma \ref{lem:1.1}}
\begin{multline}\label{eq:3.12} 
\int_{|\delta_n| \geqslant R}
f(x,\delta_n(x)) \delta_n(x) \, dx \leq \left( \int_{|\delta_n|
\geqslant R} \left| \frac{f(x,\delta_n(x))}{\delta_n(x)} \right|^\tau
dx \right)^{1/\tau} \left( \int_{|\delta_n| \geqslant R}
|\delta_n(x)|^\sigma \, dx \right)^{2/\sigma} \\ 
\leq \left(
\int_{|\delta_n| \geqslant R} \mathscr{F}(x,\delta_n(x)) \, dx
\right)^{1/\tau} \|\delta_n\|_{L^\sigma(\mathbb{R}^N)}^2 \leq
N_0^{1/\tau} \|\delta_n\|_{L^\sigma(\mathbb{R}^N)}^2.
\end{multline} 
We want to estimate the last norm of $\delta_n$ in
terms of the norm in $\mathscr{H}_\lambda$. To do this, we pick $\nu
\in (\sigma,2_\alpha^*)$ and interpolate:
\begin{multline*} \|\delta_n\|_{L^\sigma(\mathbb{R}^N)}^\sigma
\leq
\|\delta_n\|_{L^2(\mathbb{R}^N)}^{\frac{2(\nu-\sigma)}{\nu-2}}
\|\delta_n\|_{L^\nu(\mathbb{R}^N)}^{\frac{\nu (\sigma-2)}{\nu-2}}
\leq \left( \frac{1}{\lambda b} \right)^{\frac{\nu-\sigma}{\nu
-2}} \|\delta_n\|_\lambda^{\frac{2(\nu-\sigma)}{\nu-2}} \left(
\gamma_0 \gamma_\nu \|\delta_n\|_\lambda
\right)^{\frac{\nu-\sigma}{\nu -2}} + o(1) \\ \leq \left(
\gamma_0 \gamma_\nu \right)^{\frac{\nu-\sigma}{\nu -2}} \left(
\frac{1}{\lambda b} \right)^{\frac{\nu-\sigma}{\nu -2}}
\|\delta_n\|_\lambda^\sigma + o(1),
\end{multline*} where we have used \eqref{eq:3.10}. Going back to
\eqref{eq:3.12}, for a suitable positive constant $C$,
\begin{equation}\label{eq:3.13} \int_{|\delta_n| \geq R}
f(x,\delta_n(x)) \delta_n(x) \, dx \leq \left(\frac{C}{\lambda
b}\right)^{\frac{2(\nu-\sigma)}{\sigma(\nu-2)}} \|\delta_n\|_\lambda^2
+o(1).
\end{equation} On the other hand,
\begin{equation} \label{eq:3.14} \int_{|\delta_n| \leq R}
f(x,\delta_n(x)) \delta_n(x) \, dx \leq \int_{|\delta_n|
\leq R} \left( \varepsilon+ C_\varepsilon R^{q-2} \right)
|\delta_n(x)|^2 \, dx \leq \frac{C_\varepsilon R^{q-2}}{\lambda
b} \|\delta_n\|_\lambda^2 + o(1).
\end{equation} Combining now \eqref{eq:3.13} with \eqref{eq:3.14} we
obtain
\begin{multline*} o(1) = D\Phi_\lambda (\delta_n)[\delta_n] =
\|\delta_n\|_\lambda^2 - \int_{\mathbb{R}^N} f(x,\delta_n(x))
\delta_n(x)\, dx - \mu \int_{\mathbb{R}^N} \xi(x) |\delta_n(x)|^p \,
dx \\ \geq \left( 1 - C \left( \frac{1}{\lambda b} -
\left(\frac{1}{\lambda b}
\right)^{\frac{2(\nu-\sigma)}{\sigma(\nu-2)}} \right) \right)
\|\delta_n\|_\lambda^2 +o(1).
\end{multline*} It now suffices to choose $\Lambda>0$ so large that
the last bracket is strictly positive for every $\lambda \geq
\Lambda$, and we deduce that $\delta_n =o(1)$ as $n \to +\infty$.
\end{proof} We can now prove the main result of this section.
\begin{proof}[Proof of Theorem \ref{th:3.1}] First of all, we fix
$\mu_0$ such that $\Phi_\lambda$ has the mountain-pass geometry, see
Lemma \ref{lem:2.4}. Now we can introduce the value
	\[ c_\lambda = \inf_{\gamma \in \Gamma_\lambda} \sup_{t \in
[0,1]} \Phi_\lambda(\gamma(t)),
	\] where $\Gamma_\lambda = \left\{ \gamma \in
C([0,1],\mathscr{H}_\lambda) \mid \gamma(0)=0, \ \gamma(1)=e
\right\}$. By Lemma \ref{lem:3.5} there exists a number $\lambda_0>0$
such that $\Phi_\lambda$ satisfies the Palais-Smale condition at level
$c_\lambda$ for any $\lambda \geq \lambda_0$. Hence a first
solution to equation \eqref{eq:1} arises as a mountain-pass point at
level $c_\lambda>0$.
	
	To construct the second solution, we remark that there always
exists a function $\phi_0 \in \mathscr{H}_\lambda$ such that
$\int_{\mathbb{R}^N} \xi(x) |\phi_0(x)|^p \, dx >0$. On the straight
half-line $t \mapsto t \phi_0$, we have
	\[ \Phi_\lambda(t \phi_0) = \frac{t^2}{2} \|\phi_0\|_\lambda^2
- \int_{\mathbb{R}^N} F(x,t\phi_0(x))\, dx - \frac{\mu t^p}{p}
\int_{\mathbb{R}^N} \xi(x) |\phi_0(x)|^p \, dx.
	\] Since $F$ is non-negative and $1<p<2$, there exists $t_0$
close to zero (without loss of generality we assume that $t <
\varrho$) such that $\Phi_\lambda(t_0 \phi_0)<0$. On the contrary, we
already know that $\Phi_\lambda(u) >0$ if $\|u\|_\lambda =
\varrho$. For
	\[ m_\lambda = \inf \left\{ \Phi_\lambda (u) \mid u \in
\overline{B(0,\varrho)} \right\} <0
	\] there exists a sequence $\{v_n\}_n$ in
$\mathscr{H}_\lambda$ such that $\Phi_\lambda(v_n) \to
m_\lambda<0$. From the previous discussion it is not restrictive to
assume that $v_n$ is far from the boundary of
$\overline{B(0,\varrho)}$. Hence the Ekeland Variational Principle
implies that we may assume without loss of generality that
$D\Phi_\lambda(v_n) =o(1)$ as $n \to +\infty$.
	
	Taking as usual $\lambda$ large and $\mu$ small enough, the
Palais-Smale condition is satified at level $m_\lambda$, so that we
may assume $v_n \to v_0$ strongly in $\mathscr{H}_\lambda$. Then $v_0$
is another solution of \eqref{eq:1} at level $m_\lambda<0$, and proof
is complete.
\end{proof}
\begin{remark}
Theorem \ref{th:3.9} applies to $u$ and $v_0$ as well, so that our solutions are more regular.
\end{remark}

\end{document}